\documentclass[11pt,reqno]{article}
\usepackage{amsfonts,amsmath,amsthm}
\usepackage[colorlinks, linkcolor=blue, citecolor=blue, urlcolor=blue]{hyperref}

\newtheorem{thm}{Theorem}[section]%
\newtheorem{Proposition}[thm]{Proposition}%
\newtheorem{Lemma}[thm]{Lemma}%

\newtheorem{defi}[thm]{Definition}%
\newtheorem{Remark}[thm]{Remark}%
\numberwithin{equation}{section}

\newcommand{\upchi}{\raise1pt\hbox{$\chi$}}
\newcommand{\R}{{\mathord{\mathbb R}}}
\newcommand{\C}{{\mathord{\mathbb C}}}

\newcommand{\norm}[1]{\left\Vert#1\right\Vert}

\newcommand{\Bb}{{\mathcal B}}
\newcommand{\Cc}{{\mathcal C}}

\newcommand{\Kk}{{\mathcal K}}
\newcommand{\Gg}{{\mathcal G}}
\newcommand{\Ss}{{\mathcal S}}
\newcommand{\Tt}{{\mathcal T}}
\def\W{\widetilde{W}}

\def\tnm{{|\!|\!|}}

\begin{document}

\title{{\sc Decay Rates for a class of \\ Diffusive-dominated Interaction Equations}}

\author{
\vspace{5pt} Jos\'e A.  Ca\~nizo$^{1}$, Jos\'e A. Carrillo$^{2}$, and Maria E. Schonbek$^{3}$ \\
\vspace{5pt}\small{$^{1}$ Departament de Matem\`atiques}\\[-8pt]
\small{Universitat Aut\`onoma de Barcelona, E-08193 Bellaterra, Spain}\\[-4pt]
\small{Email: \texttt{canizo@mat.uab.es}}\\
\vspace{5pt}\small{$^{2}$ Instituci\'o Catalana de Recerca i Estudis Avan\c cats and Departament de Matem\`atiques}\\[-8pt]
\small{Universitat Aut\`onoma de Barcelona, E-08193 Bellaterra,
  Spain}\\[-4pt]
\small{\textit{On leave from:} Department of Mathematics, Imperial College London, London SW7
2AZ, UK.}\\[-4pt]
\small{Email: \texttt{carrillo@mat.uab.es}}\\
\vspace{5pt}\small{$^{3}$ Department of Mathematics}\\[-8pt]
\small{UC Santa Cruz, Santa Cruz, CA 95064, USA}\\[-4pt]
\small{Email: \texttt{schonbek@ucsc.edu}}
}

\date{June 19, 2011}

\maketitle

\begin{abstract}
  We analyse qualitative properties of the solutions to a mean-field
  equation for particles interacting through a pairwise potential
  while diffusing by Brownian motion. Interaction and diffusion
  compete with each other depending on the character of the
  potential. We provide sufficient conditions on the relation between
  the interaction potential and the initial data for diffusion to be
  the dominant term. We give decay rates of Sobolev norms showing that
  asymptotically for large times the behavior is then given by the
  heat equation. Moreover, we show an optimal rate of convergence in the
  $L^1$-norm towards the fundamental solution of the heat equation.
\end{abstract}

\medskip

\centerline{Mathematics subject classification numbers: 35B35,
35Q30, 76D05}


\section{Introduction}

In this paper we consider the diffusive aggregation equations,
\begin{subequations}
  \label{AE}
  \begin{align}
    \label{AE1}
    &\partial_t \rho
    = \nabla\cdot( \rho (\nabla W \ast\rho)) +\Delta \rho
    \\
    &\rho(0,x) =\rho_0,
    \label{AE2}
  \end{align}
\end{subequations}
where $\rho = \rho(t,x)$ is a real function depending on time $t \geq
0$ and space $x \in \R^N$, $W:\R^N\longrightarrow \R$ is an
\emph{interaction potential} verifying $W(x) = W(-x)$ (without loss of
generality, see \cite{CMV}). These equations have received a lot of
attention in the recent years because of their ubiquity in different
models and areas of applied and pure mathematics. Collective behavior
of animals (swarming), chemotaxis models, and granular media models
are some examples, see
\cite{ME,MCO,BDP,BenedettoCagliotiCarrilloPulvirenti98,CMV} and the
references therein. On the other hand, these equations have been
studied in connection to entropy-entropy dissipation techniques,
optimal transport, and gradient flows with respect to probability
measure distances, see \cite{CMV,AmbrosioGigliSavare02p} and the
references therein.

Without diffusion, the continuity equation \eqref{AE1} with a
singular interaction potential $W$ can lead to very involved
dynamics where blow-up can occur, and where Dirac Delta
singularities and smooth parts of the solution can coexist. More
precisely, assume that we have an interaction potential which is
radial, smooth away from the origin and whose gradient may be
singular at the origin, with a local behavior not worse than that
of $|x|$ (i.e., having at worst a Lipschitz singularity at the
origin); then, blow-up in finite time of $L^1 \cap L^\infty$
solutions was reported in \cite{BCL}. In fact, the almost sharp
condition which determines the behavior of global existence or
blow-up in $L^1 \cap L^\infty$ is the so-called \emph{Osgood
condition}. It is given in terms of the size of the gradient of
the potential $W$, which is said to satisfy the Osgood condition
if
\begin{equation}
  \label{eq:osgood}
  \int_0^1 \frac{1}{k'(r)} dr = +\infty
\end{equation}
with $W(x)=k(|x|)$. Specifically if \eqref{eq:osgood} is
satisfied, with some mild additional monotonicity conditions on
$k''$, and also $\nabla W\in W^{1,q}$ with $q<N$ and $ \rho_0 \in
L^1 \cap L^p$ with $p > N/(N-1)$, $N\geq 2$, then there are global
weak solutions, see \cite{BLR}. Aggregation (blow-up) only happens
when $t=\infty$, see \cite{BCL,BLR} for the $L^1$-$L^\infty$ and
$L^1$-$L^p$ results, respectively. In fact, weak measure solutions
were proved to exist after the $L^\infty$-blow-up time in a unique
way for certain attractive non-Osgood potentials, see \cite{CDFLS}
for a definition of weak measure solution and further details. In
\cite{CDFLS}, the authors also illustrate the existence of weak
measure solutions with very complicated patterns and Dirac Delta
formations.

In this work we address the following issue: under which
conditions on the interaction potential can linear diffusion prevail,
leading to a diffusive-dominated behavior for large times? More
precisely, we give sufficient conditions on the interaction
potential and the initial data such that the competition between the
possible aggregation due to an attractive interaction potential and
the linear diffusion in \eqref{AE1} is won by the latter.

Sharp conditions for separating global existence of solutions from
blow-up of solutions to \eqref{AE} have been given in several
papers related to homogeneous interaction potentials or to the
classical Keller-Segel model \cite{BDP,BCL2,BRB}. Also, conditions
for global existence or blow-up have been given in \cite{KS2}
based on the $L^p$ regularity of the gradient $\nabla W$ of the
potential. Blow-up conditions have been studied in detail as well
for fractional diffusions \cite{LR1,LR2}. However, we are not
aware of many results dealing with the asymptotic behavior once
the diffusion dominates over the aggregation except for
\cite{BDE,B} where the authors show that the solutions of the
Keller-Segel model behave like the solution of the heat equation
for small mass and the recent papers \cite{KS1,KS2} in which they
deal precisely with this issue for the equation \eqref{AE}; we
comment further on these works below.

Here we show that under suitable smallness conditions involving both
the interaction potential and the initial data, see Theorem
\ref{th:L2decay}, the behavior of the solution is determined by the
heat equation for large times (see Theorems \ref{th:L2decay},
\ref{thm:Hm-decay} and \ref{th:asymptotic-main}). In other words, we
get a result of asymptotic simplification for all dimensions under
some size conditions where the nonlinearity disappears and the decay
rates and behavior are like the diffusive equation at least up to
first order. We remark that the asymptotic simplification result in
$L^1$ without rate and the decay rates in $L^p$ were obtained in the
one-dimensional case in \cite{KS1} under some smallness condition
similar to one of the possibilities in Theorem \ref{th:L2decay} below
by using scaling arguments. Also, global existence results were
reported in the multidimensional case in \cite{KS2} but no
uniform-in-time $L^\infty$ bounds nor decay rates under smallness size
conditions were proved.

Asymptotic simplification results have been reported for problems
in fluid mechanics \cite{AS,ORS,SS,GW}, in convection-diffusion
equations \cite{CF}, and in some nonlinear diffusion models of
Keller-Segel type for small mass \cite{B,LS}. Here, we use the
technique of Fourier splitting \cite{S}, a technique quite
successful for the $N \geq 3$ dimensional Navier-Stokes equations,
together with direct estimates over the bilinear integral term
associated to \eqref{AE1} via Duhamel's formula to get the optimal
decay rates in Sobolev and $L^p$ spaces in section 3. Section 2 is
devoted to setting the basic well-posedness theory of
global-in-time solutions with uniform-in-time $L^\infty$
estimates. Finally, Section 4 is devoted to combining these time
decay estimates with entropy-entropy dissipation arguments
\cite{To,To99,AMTU,CT0,CT,CF,CMV,B} to obtain decay rates in
entropy in self-similar variables and in $L^1$ towards the
self-similar heat kernel for large times.


\section{Well-posedness and global bounds}

\subsection{Notation and preliminaries}
\label{sec:notation}

We usually omit the variables of the unknown $\rho$ in eq. \eqref{AE},
which are understood to be $(t,x)$. Also, we usually write $\rho_t(x)
= \rho(t,x)$, which is useful when referring to the function $x
\mapsto \rho(t,x)$ for a given time $t$ (we emphasize that $\rho_t$ is
not to be confused with $\partial_t \rho$; subindex notation for
partial derivatives is never used in this paper). When not specified,
integrals are over all of $\R^N$, and in the variable $x$.

We use the standard multi-index notation for derivatives throughout:
for a function $f:\R^N\to \C$ and a multi-index
$\gamma=(\gamma_1,\gamma_2,...,\gamma_N)$, with integers $\gamma_j
\geq 0$, we denote $|\gamma| = \sum_{j=1}^N \gamma_j$ and define
\[
\partial^\gamma f
= \partial_{x_1}^{ \gamma_1}\partial_{x_2}^{ \gamma_2}
\dots \partial_{x_N}^{ \gamma_N} f.
\]
We let $\Ss = \Ss(\R^N)$ be the usual Schwartz space of rapidly
decreasing functions. The Fourier and inverse Fourier transform of
$f\in \Ss$ are defined by
\begin{equation*}
  \hat f(\xi) = (2\pi)^{-N/2}\int_{\R^N} e^{- i x\cdot\xi} v(x) \, dx
  \qquad\mbox{and}\qquad
  \check{f}(x) = (2\pi)^{-N/2}\int_{\R^N}e^{ i x\cdot\xi}f(\xi)\, d\xi,
\end{equation*}
respectively, and extended as usual to $\mathcal{S}'$. If $k$ is a
nonnegative integer, $\mathcal{W}^{k,p}(\R^N) = \mathcal{W}^{k,p}$
will signify, as is standard, the Sobolev space consisting of
functions in $L^p(\R^{N})$ whose generalized derivatives up to order
$k$ belong to $L^p(\R^N) = L^p $, $1\leq p \leq \infty$, with norm
$\|\cdot\|_{k,p}$ defined by
\begin{equation}
  \label{eq:Wkp-norm}
  \| f \|_{k,p}^p := \sum_{|\gamma| \leq k} \|\partial^\gamma f\|_{p}^p
\end{equation}
for any $f \in \mathcal{W}^{k,p}$, where the sum is over all
multi-indices $\gamma$ with $|\gamma| \leq k$. When $p=2$,
$\mathcal{W}^{k,2}(\R^N)=H^k(\R^N) = H^k$, where the space $H^k$ is
defined for all $k\in \R$ as the space of all $f\in \mathcal{S}'$ such
that $(1+|\xi|^2)^{k/2}\hat f(\xi)\in L^2$. The norm in $H^k$ (that of
$\mathcal{W}^{k,2}$ defined in \eqref{eq:Wkp-norm}) is sometimes
denoted by $\|\cdot\|_{H^k}$ instead of $\|\cdot\|_{k,2}$. We also
use the notation
\begin{equation}
  \label{eq:Wkp-seminorm}
  \| \mathcal{D}^k f \|_p
  :=
  \Big( \sum_{|\gamma| = k} \|\partial^\gamma f\|_{p}^p \Big)^{\frac{1}{p}}.
\end{equation}
The space of bounded continuous functions from an interval $I$ to a
normed space $X$ is denoted by $BC(I,X)$. To simplify the notation,
the domain $\R^N$ of the above spaces will usually be omitted.

We will use the well-known fact that for any $m \geq 0$ there
is a constant $C$ depending only on $m$ and the dimension $N$ such
that
\begin{equation}
  \label{eq:Hm-norm-of-product}
  \| f g \|_{m,p} \leq C \big(
    \|f\|_{m,p} \|g\|_\infty +
    \|f\|_\infty \|g\|_{m,p}
  \big),
\end{equation}
for functions $f,g \in \mathcal{W}^{m,p} \cap L^\infty$.
We also define $D^m\rho$ by $\widehat{D^m \rho} =|\xi|^m
\widehat{\rho}$. Hence, by Plancherel's identity $\|D^m \rho\|_2 =
\||\xi|^m \widehat{\rho}\|_2$. By $C$ we denote arbitrary constants
that can change from line to line.

We recall the Ga\-gliar\-do-Ni\-ren\-berg-So\-bo\-lev (GNS)
inequalities, which we will be using repeatedly in the sequel, see
\cite[Theorem 9.3]{FriedBook} for a proof: given $1 \leq q,s \leq
\infty$ and integers $0\leq j < m$, there exists a number $C > 0$
depending on $q$, $s$, $j$, $m$ and the dimension $N$ such that
\begin{equation}\label{eq:GNS-Friedman}
  \| \mathcal{D}^j v \|_{p}
  \leq
  C\,\| \mathcal{D}^m v \|_{q}^\theta
  \, \| v \|_s^{1-\theta}.
\end{equation}
where we use the notation \eqref{eq:Wkp-seminorm} with
\begin{equation*}
   \frac{1}{p}
   =
   \frac{j}{N}
   + \theta \left( \frac{1}{q} - \frac{m}{N} \right)
   + (1-\theta) \frac{1}{s},
\end{equation*}
where $\frac{j}{m} \leq \theta \leq 1$, with the following
exception: if $m-j-N/q$ is a nonnegative integer, then the GNS
inequality \eqref{eq:GNS-Friedman} is only valid for $\frac{j}{m}
\leq \theta < 1$. Given that the seminorms $\|D^m v\|_p$ and
$\|\mathcal{D}^m v\|_p$ are equivalent for $1 < p < +\infty$, we
will often interchange them.

We will make use of the standard heat semigroup $e^{t\Delta}$,
which is defined as the convolution in the $x$ variable with the
heat kernel
\begin{equation}
  \label{eq:heat-kernel}
  G(t,x) := \frac{1}{(4 \pi t)^{N/2}} e^{-\frac{|x|^2}{4t}}\, ,
\end{equation}
whose derivatives satisfy for all integers $m\geq 1$ and any
multi-index $\gamma$ with $|\gamma| = m$
\begin{equation}\label{eq:heat-kernel-dx-bound}
  |\partial^\gamma G(t,x)| \leq
  C \,t^{-\frac{N+m}{2}}
  e^{-\frac{|x|^2}{8t}}
\end{equation}
for some constant $C > 0$ depending on the dimension $N$. From
\eqref{eq:heat-kernel-dx-bound}, we get the following standard
estimates:
  \begin{equation}
    \label{eq:norm-heat-gradient}
    \| \nabla G(t,\cdot) \|_r
    \leq
    C \, t^{-\frac{1}{2} (1+N(1-\theta))}
  \end{equation}
with $\theta = \frac{1}{r}$ and $C=C(N)$ and $1\leq r\leq \infty$.

\subsection{Local-in-time existence}

In order to show short-time existence of solutions to the equation
\eqref{AE} we use a common fixed-point iteration as also done in
\cite{KS2}. Some elements in the proofs of subsections 2.2 and 2.3 are
related to results reported in \cite{KS2} but we prefer to include
them for the sake of the reader. One can formally rewrite \eqref{AE}
by using Duhamel's formula and integrating by parts:
\begin{equation}
  \label{eq:Duhamel}
  \rho_t = e^{t \Delta} \rho_0
  - \int_0^t \nabla e^{(t-s) \Delta}
  \big( \rho_s (\nabla W * \rho_s) \big)
  \,ds,
\end{equation}
where $\nabla e^{t\Delta}$ denotes the convolution in $x$ with the
gradient of the heat kernel \eqref{eq:heat-kernel}. As it is common,
we define a mild solution of \eqref{AE} as one that has some
reasonable regularity for \eqref{eq:Duhamel} to make sense, and
satisfies \eqref{eq:Duhamel}. For this definition, and for the
short-time existence result \ref{th:short-time}, we do not assume that
$\rho$ is nonnegative, as it is not needed for the argument. In the
sequel we  work in any   dimension $N \geq 1$.
\begin{defi}
  Take $T \in (0,+\infty]$, $p \in [1,+\infty]$, and $\rho_0 \in
  L^p$. Assume that $\nabla W \in (L^q)^N$. A \emph{mild $L^p$
    solution} to equation \eqref{AE} on $[0,T)$ with initial condition
  $\rho_0$ is a function $\rho \in L^1_{loc}([0,T), L^p)$ such that
  \eqref{eq:Duhamel} holds for all $t \in (0,T)$.
\end{defi}

\begin{Remark}
  Note that $\nabla W * \rho \in L^1_{loc}([0,T), (L^\infty)^N)$, so if $\rho  \in L^p$ the
  product $\rho (\nabla W * \rho)$ is in $\in L^1_{loc}([0,T),
  (L^p)^N)$ and the integral in \eqref{eq:Duhamel} makes sense.
\end{Remark}

\begin{Remark}
  \label{rmk1}
  It is easy to see that whenever a mild solution $\rho$ has enough
  regularity to be a classical solution, it is then in fact a
  classical solution. This is: if a mild solution $\rho$ has
  continuous first-order time derivatives and continuous second-order
  space derivatives, then it is a classical solution.
\end{Remark}

We will show the following short-time existence theorem:
\begin{thm}[Short-time existence]
  \label{th:short-time}
  Take $p \in [1,+\infty]$, $m \geq 0$, and $\rho_0 \in
  \mathcal{W}^{m,p}$. Assume that $\nabla W \in (L^q)^N$, with
  $\frac{1}{p}+\frac{1}{q}=1$. Then there exists a maximal time $T^*
  \in (0,+\infty]$ and a unique mild solution $\rho \in
  \mathcal{C}([0,T), \mathcal{W}^{m,p})$ of problem \eqref{AE}. If
  $T^* < +\infty$ then
  \begin{equation}
    \label{eq:Wmp-blowup}
    \| \rho_t \|_{m,p} \to +\infty
    \quad \text{ as } t \to T*.
  \end{equation}

  Let $p' \in [0,+\infty]$, $m' \geq 0$. If additionally $\rho_0
  \in \mathcal{W}^{m',p'}$ and $\nabla W \in (L^{q'})^N$ with
  $\frac{1}{p'}+\frac{1}{q'}=1$, then the solution given above belongs to
  $BC([0,T), \mathcal{W}^{m,p} \cap \mathcal{W}^{m',p'})$.
\end{thm}

This result will follow from a standard fixed point theorem for
bilinear forms which  for completeness we state here, see \cite{C}:

\begin{Lemma}
  \label{lm:FPT}
  Let X be an abstract Banach space with norm $\|\cdot\|_X$ and $B : X
  \times X \to X $ a bilinear operator such that for any $x_1, x_2 \in
  X$,
  \begin{equation}
    \label{bilinear0}
    \|B(x_1,x_2)\|_X \leq \eta \|x_1\|_X\|x_2\|_X
  \end{equation}
  then for any  $y \in X$ such that
  \begin{equation}
    \label{data}
    4 \eta \|y\|_X < 1
  \end{equation}
  the equation $x = y + B(x, x)$ has a solution $x \in X$. In
  particular the solution satisfies $\|x\|_X \leq 2\|y\|_X$ and is
  the only one such that $\|x\|_X < \frac{1}{2\eta}$.
\end{Lemma}

\begin{proof}[Proof of Theorem \ref{th:short-time}]
  Take $T > 0$. Following a standard strategy, we first show that for
  $T$ small enough there exists a mild solution on $[0,T)$. In order to find
  a function $\rho$ satisfying \eqref{eq:Duhamel} we choose the
  bilinear form defined as
  \begin{equation} \label{bilinear}
    \Bb(\rho,\psi)
    =
    - \int_0^t \nabla e^{(t-s)\Delta}
    \cdot (\nabla W \ast \psi_s) \rho_s \,ds
    \quad \text{ for } t \in [0,T).
  \end{equation}
  In order to apply Lemma \ref{lm:FPT} estimate this bilinear form in
  the space $X := BC([0,T), \mathcal{W}^{m,p})$, with norm given by
  \begin{equation*}
    \tnm \rho \tnm_{m,p} := \sup_{t \in [0,T)} \|\rho_t\|_{m,p}.
  \end{equation*}
  In this proof we denote by $C_{m}$ a number that depends only on
  $m$, which may change from line to line. If $\gamma$ is a
  multi-index with $|\gamma| \leq m$, then, for $\rho, \psi \in
  BC([0,T), \mathcal{W}^{1,p})$, and any $t \in [0,T)$, using
  \eqref{eq:norm-heat-gradient}, we get
  \begin{align}
    \| \partial^\gamma \Bb(\rho, \psi) \|_{p}
    &\leq
    \int_0^t \Big\| \nabla e^{(t-s)\Delta} \cdot
    \partial^\gamma \big( (\nabla W \ast \psi_s) \rho_s \big)
    \Big\|_p \,ds\leq C
    \int_0^t \frac{1}{\sqrt{t-s}}
    \Big\| \partial^\gamma \big( (\nabla W \ast \psi_s) \rho_s \big)
    \Big\|_p \,ds\nonumber
    \\
    &\leq C_{m}
    \int_0^t \frac{1}{\sqrt{t-s}}
    \big\| \nabla W * \psi_s \big\|_{m,\infty}
    \, \big\| \rho_s \big\|_{m,p}
    \,ds\nonumber
    \\
    &\leq C_{m}
    \int_0^t \frac{1}{\sqrt{t-s}}
    \|\nabla W\|_q \|\psi_s\|_{m,p}
    \, \| \rho_s \|_{m,p}
    \,ds,     \label{eq:W1p-estimate}
  \end{align}
  which shows, taking the supremum on $[0,T)$ and summing over all
  multi-indices with $|\gamma| \leq m$, that
  \begin{equation*}
    \tnm \Bb(\rho,\psi) \tnm_{m,p}
    \leq
    C_{m} \sqrt{T} \,
    \|\nabla W\|_q \,
    \tnm \rho_s \tnm_{m,p}
    \, \tnm \psi_s \tnm_{m,p}.
  \end{equation*}
  In the third inequality in \eqref{eq:W1p-estimate} we have used
  that, for $f \in \mathcal{W}^{m,p}$, $g \in \mathcal{W}^{m,\infty}$,
  \begin{equation}
    \label{eq:bilinear-bound}
    \| \partial^\gamma (f g) \|_p
    \leq
    C_{m} \| f \|_{m,p} \, \| g \|_{m,\infty},
  \end{equation}
  for any multi-index with $|\gamma| \leq m$. This can be easily seen
  by writing out and estimating the derivatives of the product
  $fg$. This gives the estimate \eqref{bilinear0} with $\eta = C_{m}
  \sqrt{T} \, \|\nabla W\|_q$. Taking $y \in BC([0,T),
  \mathcal{W}^{m,p})$ defined by $t \mapsto e^{t\Delta} \rho_0$ in
  Lemma \ref{lm:FPT}, we can choose $T$ small enough so that
  \eqref{data} is satisfied. Hence Lemma \ref{lm:FPT}, yields the
  existence of a function $\rho \in BC([0,T), \mathcal{W}^{1,p})$
  satisfying \eqref{eq:Duhamel}, i.e., a mild solution. This solution
  is a priori unique only in the set of solutions satisfying $\tnm
  \rho \tnm_{m,p} \leq 1/(2 \eta)$, but a standard argument using the
  continuity of $\rho$ shows that it is in fact the unique mild
  solution in $BC([0,T), \mathcal{W}^{1,p})$.

  The existence of a maximal time $T^*$ and the blow-up of the
  solution at $T^*$ if $T^* < +\infty$ follows now from a standard
  argument. The last part of Theorem \ref{th:short-time} is obtained
  by an analogous reasoning, considering now the space $X := BC([0,T),
  \mathcal{W}^{m,p} \cap \mathcal{W}^{m',p'})$.
\end{proof}

\subsection{ A priori time-dependent bounds and global existence}

In this section we obtain $L^p$ bounds for our solutions. Due to
equation \eqref{eq:Wmp-blowup} in Theorem \ref{th:short-time} these
bounds imply global existence of the solutions.

\begin{Proposition}
  \label{pr:Lp-bounds} \
  \begin{itemize}
  \item[i)] Let $\rho_0 \in L^1 \cap L^p$ with $\rho_0 \geq 0$, and
    $\nabla W \in (L^q \cap L^{\infty})^N$ with $1/p + 1/q = 1$. Let
    $\rho \in \mathcal{C}([0,T); L^1 \cap L^p)$ be a mild solution to
    \eqref{AE} on $[0,T)$ as obtained in Theorem
    {\rm\ref{th:short-time}} with initial data $\rho_0$. Then there is a
    constant $C \geq 0$ (depending only on $\rho_0$ and the dimension
    $N$) such that for all $t\in [0,T)$
    \begin{align}
      \label{L1}
      &\|\rho_t\|_1 = \|\rho_0\|_1 =: M,
      \\
      \label{lotro}
      &\|\rho_t\|_{p}
      \leq
      \|\rho_0\|_{p}
      \exp \left(
        C \|\nabla W\|_{\infty} \|\rho_0\|_1
        \sqrt{t}
      \right)
      := C_p(t) \, \|\rho_0\|_{p}.
    \end{align}
  \item[ii)] Assume that $\rho_0 \in \mathcal{W}^{m,2} \cap L^1 \cap L^\infty
    $ for some $m \geq 1$, and $\nabla W \in (L^1 \cap
    L^\infty)^N$. Then there is a time-dependent function $C_{m,p}(t)$,
    bounded on finite time intervals, depending only on $\rho_0$,
    $\|\nabla W\|_1$, $\|\nabla W\|_\infty$ and the dimension $N$, such
    that for all $t \in [0,T)$,
    \begin{equation}
      \label{eq:Hm-bound}
      \|\rho_t\|_{H^m}
      \leq
      C_{m,p}(t) \, \|\rho_0\|_{H^m}.
    \end{equation}
  \end{itemize}
\end{Proposition}

\medskip

To prove this proposition we will use the following modified Gronwall
Lemma (see \cite{SS}):

\begin{Lemma} \label{modgron}
  Let $0\leq T \leq \infty,\,\delta\in(0,1)$,   and let $f:[0,T] \to
  [0,\infty)$ be continuous and satisfy
  \[
  f(t) \leq A +B\int_0^t (t-s)^{-\delta} f(s) \;ds
  \]
  for all $t\in[0,T)$. Then
  \[
  f(t) \leq A \Phi (B\Gamma(1-\delta)t^{1-\delta})
  \]
  for $t\in[0,T)$, where $\Phi :\C \to \C$ is defined by
  \[
  \Phi(z) = \sum_{n=0}^{\infty} \frac{z^n}{\Gamma(n(1-\delta)+1)}
  \]
\end{Lemma}

\begin{proof}[Proof of Proposition \ref{pr:Lp-bounds}]
  Inequality (\ref{L1}) follows by direct integration of equation
  \eqref{AE} (or of the mild formulation (\ref{eq:Duhamel})) since it
  has divergence form. To obtain (\ref{lotro}), take the $L^p$ norm in
  \eqref{eq:Duhamel} and use \eqref{eq:norm-heat-gradient} to obtain
  \begin{align}
    \label{cosas}
    \|\rho_t \|_p
    &\leq \| \rho_0 \|_p
    + C \int_0^t \frac{1}{\sqrt{t-s}}\|\rho_s\|_p
    \,\| \nabla W * \rho_s \|_{\infty} \,ds
    \\
    \notag
    &\leq \| \rho_0 \|_p
    + C \int_0^t \frac{1}{\sqrt{t-s}}\|\rho_s\|_p
    \, \| \nabla W \|_{\infty}
    \| \rho_s \|_{1} \,ds
    \\
    \notag
    &\leq \| \rho_0 \|_p
    + C \| \nabla W \|_{\infty} \|\rho_0\|_1
    \int_0^t
    \frac{1}{\sqrt{t-s}}\|\rho_s\|_p \,ds.
  \end{align}
  The modified Gronwall inequality (Lemma \ref{modgron}) yields
  (\ref{lotro}).
  To obtain (\ref{eq:Hm-bound}) take any multi-index $\gamma$ with
  $|\gamma| \leq m$, apply $\partial^\gamma$ to \eqref{AE} and multiply
  by $\partial^\gamma \rho$ to obtain
  \begin{equation}
    \label{eq:Wm-1}
    \frac{1}{2} \frac{d}{dt} \int (\partial^\gamma \rho)^2
    =
    - \int \nabla (\partial^\gamma \rho) \cdot
    \big( \partial^\gamma (\rho (\nabla W * \rho)) \big)
    - \int |\nabla \partial^\gamma \rho|^2.
  \end{equation}
  (To make this reasoning rigorous, as we are using eq. \eqref{AE}
  instead of the weak formulation, we have to carry it out on
  approximating solutions with smooth initial data and then pass to
  the limit. This process is straightforward and as such we omit the
  details.) By Cauchy-Schwarz's inequality the first term can be
  bounded by
  \begin{equation}
    \label{eq:Wm-2}
    \left|
      \int \nabla (\partial^\gamma \rho) \cdot
      \big( \partial^\gamma (\rho (\nabla W * \rho)) \big)
    \right|
    \leq
    \left(
      \int |\nabla \partial^\gamma \rho|^2
    \right)^{1/2}
    \left(
      \int \big| \partial^\gamma (\rho (\nabla W * \rho)) \big|^2
    \right)^{1/2}.
  \end{equation}
  To estimate the second parentheses, use equation
  \eqref{eq:Hm-norm-of-product} and then Young's convolution
  inequality to obtain
  \begin{multline*}
    \left(
      \int \big| \partial^\gamma (\rho (\nabla W * \rho)) \big|^2
    \right)^{1/2}
    \leq
    \| \rho (\nabla W * \rho) \|_{m,2}
    \leq
    C \big(
    \| \rho \|_{m,2} \|\nabla W * \rho\|_{\infty}
    +
    \| \rho \|_{\infty} \|\nabla W * \rho\|_{m,2}
    \big)
    \\
    \leq
    C \|\nabla W \|_\infty\, \| \rho \|_{m,2}
    \big(
    \|\rho\|_1 + \| \rho \|_{\infty}
    \big)
    =: C(t) \| \rho \|_{m,2},
  \end{multline*}
  with $C(t)$ a given function that involves $\|\nabla W\|_\infty$,
  $\|\rho_0\|_1$ and $C_\infty(t)$ from \eqref{lotro}. Using this in
  \eqref{eq:Wm-2} and applying Young's inequality we get
  \begin{equation}
    \label{eq:Wm-2.5}
    \left|
      \int \nabla (\partial^\gamma \rho) \cdot
      \big( \partial^\gamma (\rho (\nabla W * \rho)) \big)
    \right|
    \leq
    \frac{1}{2}
    \int |\nabla \partial^\gamma \rho|^2
    +
    C(t)^2 \| \rho \|_{m,2}^2.
  \end{equation}
  Combining this  with  in \eqref{eq:Wm-1} yields
  \begin{equation*}
    \frac{1}{p} \frac{d}{dt} \int (\partial^\gamma \rho)^2
    \leq
    C(t)^2 \| \rho \|_{m,2}^2.
  \end{equation*}
  Adding  all multi-indices $\gamma$ with $|\gamma| \leq m$ gives
    \begin{equation*}
    \frac{d}{dt} \|\rho\|_{m,2}^2
    \leq
    C(t)^2 \| \rho \|_{m,2}^2
  \end{equation*}
  for some other time-dependent function $C(t)$. Integrating this
  inequality over time proves the last part of the proposition if
  all derivatives above are well defined.
\end{proof}

\medskip

\noindent Now, we combine the short time existence in Theorem
\ref{th:short-time} with the a priori results in Proposition
\ref{pr:Lp-bounds} to yield the global existence.

\begin{thm}
  \label{global}
  Under the conditions i) of Proposition {\rm\ref{pr:Lp-bounds}},
  there exists a unique global mild solution $\rho$ of \eqref{AE} with
  $\rho \in \mathcal{C}([0,\infty);L^p)$. Under the conditions ii) of
  Proposition {\rm\ref{pr:Lp-bounds}}, there exists a unique global
  mild solution $\rho$ of \eqref{AE} with $\rho \in
  \mathcal{C}([0,\infty);H^m)$.
\end{thm}


\subsection{Uniform bound for $\|\rho\|_\infty$}

Consider the solution $\rho \in \mathcal{C}([0,+\infty), L^1 \cap
L^\infty)$ obtained from Theorem \ref{global} with $p = \infty$. In
this section we prove that the $L^\infty$ norm is actually uniformly
bounded for all times:

\begin{thm}
  \label{th:Linftybound}
  Let the interaction potential $W$ be such that $\nabla W \in (L^1\cap
  L^\infty)^N$. Let $\rho_0 \in L^1\cap L^{\infty}$ nonnegative, and
  suppose $\rho$ is the solution constructed in Theorem {\rm
    \ref{global}} with data $\rho_0$. Then there exists a constant
  $\Cc_{\infty} $ depending only on $N$, $W$, and $\|\rho_0\|_1$ such
  that
  \begin{equation}
    \label{eq:Linftybound}
    \| \rho_t \|_{\infty} \leq \Cc_{\infty}
    \quad \text{ for all } t \geq 0.
  \end{equation}
\end{thm}

\begin{proof}
  Choose any time $t_0 > 0$. From Proposition \ref{pr:Lp-bounds} there
  is a constant $\tilde{C}_0$ for which
  \begin{equation}
    \label{eq:local_bound}
    \| \rho_t \|_\infty \leq \tilde{C}_0 \, ,
    \qquad t \in [0,t_0].
  \end{equation}
  We will prove that, for some $\delta > 0$,
  \begin{equation*}
    \| \rho_t \|_\infty \leq \Cc_{\infty} \, ,
    \qquad t \in [t_0-\delta,\infty).
  \end{equation*}
  This will ensure the global-in-time bound for the $L^{\infty}$ norm.

  Pick $ \delta$ satisfying $t_0 > \delta > 0$ (further conditions on
  $\delta$ will be fixed below). As $\rho$ is a mild solution, we may
  use Duhamel's formula \eqref{eq:Duhamel} between $t-\delta$ and $t$
  to obtain, for any $t \geq \delta$,
  \begin{equation}
    \label{eq:Duhamelt-delta-t}
    \rho_t = e^{\delta \Delta} \rho_{t-\delta}
    - \int_{t-\delta}^t \nabla e^{(t-s) \Delta}
    \big( \rho_s (\nabla W * \rho_s) \big)
    \,ds.
  \end{equation}
  For the first term in \eqref{eq:Duhamelt-delta-t} we have
  \begin{align}
    \label{eq:term1-bound}
    e^{ \delta \Delta} \rho_{t-\delta}\,
    &\leq \frac{1}{(4 \pi   \delta)^{N/2}}
    \int e^{-\frac{|x-y|^2}{4t}} \rho(t-\delta, y) \,dy
    \\
    & \leq \frac{1}{(4 \pi \delta)^{N/2}} \| \rho_{t-\delta} \|_1 =
    \frac{1}{(4 \pi \delta)^{N/2}} \|\rho_0\|_1.\nonumber
  \end{align}
  To bound the second term in (\ref{eq:Duhamelt-delta-t}) let $g := (\nabla W * \rho_s) \rho_s$,
  \begin{equation}
    \label{eq:term2-bound}
    \left| \int_{t-\delta}^t \nabla e^{(t-s)\Delta} \cdot g  \,ds \right|
    \,\leq
    \int_{t-\delta}^t \| \nabla G(t-s,\cdot) * g \|_\infty \,ds \\
    \leq
    \int_{t-\delta}^t \| \nabla G(t-s,\cdot) \|_r \, \|g\|_q\,
    \,ds
  \end{equation}
  with $\frac{1}{r} + \frac{1}{q} = 1$. We use the following estimate
  for $\|g\|_q$:
  \begin{equation}
    \label{eq:bound-norm-g}
    \|g\|_q
    \leq \|\rho_s\|_\infty \| \nabla W * \rho_s \|_q
    = \|\rho_s\|_\infty \, \| \nabla W \|_q \, \| \rho_0 \|_1\,,
  \end{equation}
  and \eqref{eq:norm-heat-gradient} for $\| \nabla G(t-s,\cdot) \|_r$
  with $\theta = \frac{1}{r}$, to conclude that
  $$
    \left| \int_{t-\delta}^t \nabla e^{(t-s)\Delta} \cdot g  \,ds \right|
    \,\leq
     C \| \nabla W \|_q \, \| \rho_0 \|_1
    \int_{t-\delta}^t (t-s)^{-\frac{1}{2} (1+N(1-\theta))}
    \|\rho_s\|_\infty \,ds\,.
    $$
  To finish the
  argument it is necessary that the right-hand side of the last  expression is  integrable near $0$, hence
  we need
  \begin{equation}
    \label{eq:a}
    -1 < -\frac{1}{2} (1+N(1-\theta))
    \Leftrightarrow
    1-\theta < \frac{1}{N}
    \Leftrightarrow
    r < \frac{N}{N-1},
  \end{equation}
  which means, from \eqref{eq:a}, that we need $N<q \leq \infty$. Putting
  \eqref{eq:term1-bound}, \eqref{eq:term2-bound},
  \eqref{eq:bound-norm-g} and \eqref{eq:norm-heat-gradient} together
  in \eqref{eq:Duhamelt-delta-t} we obtain, for $t \geq \delta$,
  \begin{align*}
    \| \rho_t \|_\infty \leq \, & \frac{1}{(4 \pi \delta)^{N/2}}
    \|\rho_0\|_1 + C \| \nabla W \|_q \, \| \rho_0 \|_1
    \int_{t-\delta}^t (t-s)^{-\frac{1}{2} (1+N(1-\theta))}
    \|\rho_s\|_\infty \,ds \nonumber
    \\
    \leq\, & C\,\delta^{-N/2} \|\rho_0\|_1 + C\,C_2 \sup_{\tau \in
      (t-\delta,t)} \|\rho(\tau)\|_\infty \int_{t-\delta}^t
    (t-s)^{-\frac{1}{2} (1+N(1-\theta))} \,ds \nonumber
    \\
    \leq \, & C\, \delta^{-N/2} \|\rho_0\|_1 + C\, C_2\, \delta^{\mu}
    \sup_{\tau \in (t-\delta,t)} \|\rho(\tau)\|_\infty
  \end{align*}
  where $C_2:=\| \nabla W \|_q \, \| \rho_0 \|_1$, $C$ is a generic
  constant depending only on $N$ and $2\mu=1-N(1-\theta)>0$. Now, we
  take the supremum of both sides of the inequality for $t\geq t_0$ to
  obtain
  \begin{align*}
    \sup_{t \geq t_0}\| \rho_t \|_\infty \leq \,& C\, \delta^{-N/2}
    \|\rho_0\|_1 + C\, C_2\, \delta^{\mu}
    \sup_{t \geq t_0-\delta} \|\rho(\tau)\|_\infty \\
    \leq \,& C\, \delta^{-N/2} \|\rho_0\|_1 + C\, C_2\, \delta^{\mu}
    \, \tilde C_0 + C\, C_2\, \delta^{\mu} \sup_{t \geq t_0}
    \|\rho(\tau)\|_\infty \nonumber
  \end{align*}
  by using the local estimate in time \eqref{eq:local_bound} for times
  less than $t_0$. Choosing $0<\delta<t_0$ such that $2\,C\, C_2\,
  \delta^{\mu} = 1$ we have
  \begin{equation*}
    \label{eq:Duhamel-bound3}
    \sup_{t \geq t_0} \| \rho_t \|_\infty
    \leq 2\,C\, \delta^{-N/2} \|\rho_0\|_1 + \tilde{C}_0
    = (2\,C)^{1+N/2\mu}\,\| \nabla W \|_q^{N/(2\mu)}
    \, \| \rho_0 \|_1^{1+N/(2\mu)} + \tilde{C}_0
    :=\Cc_{\infty}.
  \end{equation*}
  This completes the proof of the theorem.
\end{proof}


\section{Algebraic decay of the solution}

\subsection{ $L^2$ decay}

To obtain the decay in $L^2$ of the solutions, we first need to show
that solutions satisfy an estimate of the form
\begin{equation} \label{energy}
  \frac{d}{dt} \int \rho^2 \,dx
  \leq
  -C \int |\nabla \rho|^2 \,dx,
\end{equation}
for some $C > 0$. This will hold for {\it sufficiently small
  potentials}. From this, given that the solution remains in $L^1$,
decay at the same rate as solutions to the heat equation will follow.

\begin{thm}
  \label{th:L2decay}
  Let the potential $W$ be such that $\nabla W \in (L^1\cap L^\infty)^N$
  and take $\rho_0 \in L^1 \cap L^{\infty} \cap H^1$
  nonnegative. Assume additionally that one of the following smallness
  conditions holds:
  \begin{enumerate}
  \item[i)]
    \label{it:rhoL1}
    It holds that $\|\nabla W\|_\infty
    \|\rho_0\|_1^{\frac{N+4}{N+2}} < 1$.
  \item[ii)]
    \label{it:rhoLinfty}
    $W \in L^1$ and $\|W\|_1 \Cc_\infty < 1$, where
    $\Cc_\infty$ is such that \eqref{eq:Linftybound} holds (i.e., $\|
    \rho_t \|_{\infty} \leq \Cc_{\infty}$ for all $t \geq 0$).

  \item[iii)]
    \label{it:rhoL2}
    $W \in L^2$ and $\|W\|_2 \Cc_2 < 1$, where $\Cc_2$ is
    such that $\| \rho_t \|_{2} \leq \Cc_{2}$ for all $t
    \geq 0$.

  \item[iv)]
    \label{it:DeltaW}
    $\Cc\|\rho_0\|_1 \|[\Delta W]_{+} \|_{N/2} \leq 1/4$ where $\Cc$
    is given below in the proof and $[\Delta W]_{+} := \max\{\Delta W,
    0\}$.
  \end{enumerate}
  Then there exists a number $\Kk > 0$ depending only on the constants
  defined in the hypotheses such that
  \begin{equation}\label{eq:L2decay}
    \|\rho_t\|_2 \leq  \Kk \, (t+1)^{-N/2}\, .
  \end{equation}
\end{thm}

\begin{Remark}
  The constant $\mathcal{C}_2$ can be estimated by interpolation as $\mathcal{C}_2\leq\|\rho_0\|_1^{1/2}\mathcal{C}_{\infty}^{1/2}$.
  However, the constant $\mathcal{C}_{\infty}$ depends both
  on $W$ and $\rho_0$ in a complicated way, although an explicit bound can be extracted from the proof of Theorem
  {\rm\ref{th:Linftybound}}.
\end{Remark}

\begin{proof}
  We first establish (\ref{energy}) in all cases. Multiply equation
  \eqref{AE1} by $\rho$ and integrate to get
  \begin{equation}
    \label{cosa}
    \frac{d}{dt} \int \rho^2 dx
    =
    \int
    \rho \nabla \cdot(\rho \nabla W\ast \rho) dx
    -\int |\nabla \rho |^2 dx.
  \end{equation}
  We need to bound the first term on the right-hand side of the last
  equation.

  \medskip

  \noindent {\bf Case i): } We use H\"older's and Young's
  inequalities to obtain
  \begin{align*}
    \int \rho \nabla \cdot(\rho \nabla W\ast \rho) \,dx
    &=- \int  \nabla \rho \cdot(\rho \nabla W\ast \rho) \,dx
    \leq \|\nabla \rho\|_2  \|\rho\, (\nabla W * \rho)\|_2
    \\
    &\leq
    \|\nabla \rho\|_2 \, \|\rho\|_2 \, \|\nabla W * \rho\|_\infty
    \leq
    \|\nabla \rho\|_2 \, \|\rho\|_2 \,
    \|\nabla W \|_\infty \, \|\rho\|_1
    \\
    &\leq
    \|\nabla \rho\|_2^2\,
    \| \nabla W \|_\infty \, \|\rho_0 \|_1^{\frac{N+4}{N+2}},
  \end{align*}
  where we interpolated the $L^2$ norm of $\rho$ between $L^1$ and
  $\dot{H}^1$ by means of the GNS inequality \eqref{eq:GNS-Friedman}
  with $j=0$, $p=q=2$ and $m=s=1$ and the conservation of mass
  \eqref{L1}. Combining this estimate with (\ref{cosa}) yields
  (\ref{energy}), due to the condition i) in the hypotheses.

  \noindent {\bf Case ii): }
  By a similar reasoning we have
  \begin{align*}
    \int \rho \nabla \cdot(\rho \nabla W\ast \rho) dx  &=- \int  \nabla \rho \cdot(\rho \nabla W\ast \rho) dx\\
    &
    \leq
    \| \nabla \rho\|_2\, \|\rho\, (W * \nabla \rho)\|_2
    \leq
    \|\nabla \rho\|_2^2 \, \|\rho\|_{\infty} \, \|W\|_1
    \leq
    \|\nabla \rho\|_2^2 \ \Cc_\infty \, \|W\|_1.
  \end{align*}
  Combining this estimate with (\ref{cosa}) yields
  (\ref{energy}), since $\Cc_\infty \, \|W\|_1 < 1$.

  \medskip

  \noindent {\bf Case iii): } Similarly,
  \begin{align*}
    \int \rho \nabla \cdot(\rho \nabla W\ast \rho) dx
    &
    = -\int  \nabla \rho \cdot(\rho \nabla W\ast \rho) dx
    \\
    &\leq
    \| \nabla \rho\|_2 \|\rho (W\ast \nabla  \rho)\|_2
    \leq
    \|\nabla \rho\|_2 \|\rho\|_2\| W \ast \nabla \rho\|_{\infty}
    \\
    &\leq \|\nabla \rho\|_2^2\ \Cc_2\, \|W\|_{2}.
  \end{align*}
  This gives \eqref{energy} as before, since $\Cc_2\, \|W\|_{2} < 1$
  by hypothesis.

  \medskip

  \noindent {\bf Case iv):} Since the Laplacian of the interaction
  potential lies in an $L^p$ space, we can write
  \begin{equation*} \label{caso3} \int \rho \nabla \cdot(\rho \nabla
    W\ast \rho) dx = -\int \rho \nabla \rho \cdot (\nabla W \ast \rho)
    dx= \int \frac{\rho^2}{2} (\Delta W\ast \rho) dx \leq \int
    \frac{\rho^2}{2} ([\Delta W]_{+}\ast \rho) =: I \, .
  \end{equation*}

  \noindent Case iv.a): We first consider the case $N=2$. Then,
  choosing any $1 < p < +\infty$,
  \begin{equation*}
    \label{est-1}
    2I \leq \|\rho\|_{2p}^2 \|[\Delta W]_{+} * \rho\|_q
    \quad \text{ with } \; \frac{1}{p} + \frac{1}{q} =1.
  \end{equation*}
  By the GNS inequality (\ref{eq:GNS-Friedman}) we have
  \begin{equation}
    \label{a}
    \|\rho\|_{2p}^2 \leq C  \|\nabla\rho\|_2^{2a}
    \|\rho\|_1^{2(1-a)}\,,
  \end{equation}
  where $a = \frac{2p-1}{2p}$. Now we estimate the convolution by
  \begin{equation} \label{b} \|[\Delta W]_{+}\ast \rho\|_q \leq
    \|[\Delta W]_{+}\|_1\|\rho\|_q\,.
  \end{equation}
  Interpolating $L^q$ between $L^1$ and $H^1$ by the GNS inequality
  (\ref{eq:GNS-Friedman}) with $b = \frac{1}{p}$ yields the estimate
  \begin{equation} \label{c} \|\rho\|_q \leq C \|\nabla\rho\|_2^{b
    }\|\rho\|_1^{1-b}\,.
  \end{equation}
  Note that $a +\frac{b}{2}=1$. Thus combining (\ref{a}), (\ref{b}),
  and (\ref{c}) gives
  \[
  I \leq 2C\|[\Delta W]_{+}\|_1 \|\nabla\rho\|_2^2 \|\rho\|_1 \, .
  \]
  Note that $p$ can be chosen to be any $p\in (1,\infty)$, and this
  final inequality does not depend on the choice (except for the
  constant $C$ in front of the inequality). This
  concludes the proof of the case $N=2$.

  \medskip

  \noindent Case iv.b): We now consider $N\geq 3$. In this case the
  nonlinear term is bounded by
  \begin{equation*}
    I \leq \frac12\|\rho\|_2^2\|[\Delta W]_{+}\ast \rho\|_{\infty}
    \leq C \|\rho\|_1^{\frac{4}{N+2}}\|\nabla\rho\|_2^{\frac{2N}{N+2}}
    \|[\Delta W]_{+}\ast \rho\|_{\infty}\,,
  \end{equation*}
  by means of the GNS inequality \eqref{eq:GNS-Friedman}. To estimate
  the term $\|[\Delta W]_{+}\ast \rho\|_{\infty}$, we use H\"older's
  inequality
  \[
  \|[\Delta W]_{+}\ast \rho\|_{\infty} \leq \|[\Delta W]_{+}\|_p
  \, \|\rho\|_{q} \, ,
  \]
  with $1 < p < +\infty$ and $\frac{1}{p} +\frac{1}{q} =1$. We
  interpolate the $L^q$ norm between $L^1$ and $\dot{H}^1$ using again
  the GNS inequality \eqref{eq:GNS-Friedman} to obtain
  \begin{equation}
    \label{eq:gnsq}
    \|\rho\|_q \leq \|\nabla\rho\|_2^a \,\|\rho\|_1^{1-a} \,,
  \end{equation}
  with $\frac{1}{q} = a (\frac{1}{2} -\frac{1}{N}) +1-a$, and thus $a=
  \frac{2N}{p(N+2)}$. For $a\leq 1$ we need $p \geq
  \frac{2N}{2+N}$. We want to choose $p$ so that
  \[
  \frac{N}{2+N} + \frac{a}{2}
  = \frac{N}{2+N} + \frac{N}{p(N+2)}
  = 1.
  \]
  This will hold if $p=N/2$, $N\geq 3$. Notice that the GNS inequality
  \eqref{eq:gnsq} does not hold for $N=2$ since we should take
  $a=1$ and $q=\infty$, which is not allowed. From the above
  inequalities, it follows that
  \[
  I \leq C \|\rho\|_1\, \|\nabla \rho\|_2^2 \,\|[\Delta
  W]_{+}\|_{N/2}\,.
  \]
  Hence, Case iv) for $N \geq 3$ follows.

  \medskip

  \noindent We  now finish the proof of the $L^2$-decay in the
  three cases above. The GNS inequality (\ref{eq:GNS-Friedman}) once more gives
  \begin{equation*}
    \|\rho\|_2 \leq \|\nabla \rho\|_2^{\frac{N}{N+2}}
    \|\rho\|_1^{\frac{2}{N+2}},
    \quad \text{ so }
    \quad
    \|\nabla \rho\|_2
    \geq \|\rho\|_2^{\frac{N+2}{N}} \|\rho\|_1^{-\frac{2}{N}}.
  \end{equation*}
  Thus from (\ref{energy}) it follows that
  \[
  \frac{d}{dt} \int \rho^2 dx \leq -\frac{1}{2} \int |\nabla \rho|^2
  dx \leq -\Kk \|\rho\|_2^{\frac{N+2}{N}}\,.
  \]
  Integration of this differential inequality yields the expected
  decay.
\end{proof}

\begin{Remark}
Some typical examples of interaction potentials in applications
are variations of the so-called Morse potential
$W(x)=1-e^{-|x|^\alpha}$ with $\alpha\geq 1$, see {\rm\cite{BCL}}.
For instance, we can get decay for solutions to the aggregation
equation with $W(x)= 1-e^{-|x|^2}$ provided
$$
\Cc \leq \frac{1}{4}\left(\|\rho_0\|_1 \|[\Delta
W]_{+}\|_{N/2}\right)^{-1},
$$
here $\|[\Delta W]_{+}\|_{N/2}^{N/2} = \int _{|x|\leq N}(|x|^2
-N)^{N/2} e^{-\frac{N}{2}|x|^2} dx$.
\end{Remark}

\subsection{ $H^m$ decay}

In this subsection we consider the decay in $H^m$ spaces. The aim is
to show that for certain potentials the decay rate will be the same as
the one for the solutions to the heat equation.  Specifically, we show
that if the $L^2$ decay happens at the same rate as for the heat
equation (as was shown in Theorem \ref{th:L2decay} under some
additional conditions) then the decay in $H^m$ will happen at the
same rate as for the heat equation. Our potential will satisfy the
hypotheses given in Theorems \ref{global} and \ref{th:Linftybound},
which ensures the existence of a unique global solution,
$L^{\infty}$-uniformly bounded in time. In this subsection, we remind
the reader that $D^m \rho$ is defined via the Fourier transform by
$\widehat{D^m \rho} =|\xi|^m \widehat{\rho}$, as remarked in Section
\ref{sec:notation}.

\medskip

We first give a technical lemma to be used later in the proof:
\begin{Lemma}
  \label{lem:Dm-product}
  Take $m \geq 1$, and functions $f, g, h \in L^2$ such that $D^m f$
  and $D^m g$ are in $L^2$. Then $D^m (f(g*h))$ is also in $L^2$ and
  \begin{equation*}
    \| D^m (f(g*h)) \|_2
    \leq
    2^{m-1} \left(
      \| D^m f \|_2 \|g\|_2 \|h\|_2 + \|f\|_2 \|D^m g\|_2 \|h\|_2
    \right).
  \end{equation*}
\end{Lemma}

\begin{proof}
  Denote $u := g*h$. Using that $ |\xi|^m \leq 2^{m-1} ( |\xi-\eta|^m
  +|\eta|^m )$ for any $\xi, \nu \in \C$,
  \begin{align*}
    \int |D^m (fu) |^2 \,dx
    &=
    \int ||\xi|^m \widehat{fu} |^2 \,d\xi
    =
    \int ||\xi|^m \widehat{f} * \widehat{u} |^2 \,d\xi
    =
    \int \left| \int
      |\xi|^{m} \widehat{f}(\nu) \widehat{u}(\xi-\nu)
      \,d\nu \right|^2 \,d\xi
    \\
    &\leq
    2^{2m-2} \int \left|
      \int
      |\nu|^{m} |\widehat{f}(\nu)| |\widehat{u}(\xi-\nu)|
      \,d\nu
      +
      \int
      |\widehat{f}(\nu)| |\xi-\nu|^{m} |\widehat{u}(\xi-\nu)|
      \,d\nu
    \right|^2 \,d\xi
    \\
    &=
    2^{2m-2} \int \left|
      (|\xi|^m |\widehat{f}|) * |\widehat{u}|
      +  |\widehat{f}| * (|\xi|^m |\widehat{u}|)
    \right|^2 \,d\xi.
  \end{align*}
  Consequently, using Young's inequality for convolutions,
  \begin{equation*}
    \| D^m (fu) \|_2
    \leq
    2^{m-1} \left(
      \big\| (|\xi|^m |\widehat{f}|) * |\widehat{u}| \big\|_2
      + \big\| |\widehat{f}| * (|\xi|^m |\widehat{u}|) \big\|_2
    \right)
    \leq
    2^{m-1} \left(
      \| D^m f \|_2 \|\widehat{u}\|_1 + \|f\|_2 \||\xi|^m\widehat{u}\|_1
    \right).
  \end{equation*}
  The proof is now finished by noticing that
  \begin{equation*}
    \|\widehat{u}\|_1 = \|\widehat{g} \widehat{h}\|_1
    \leq \|\widehat{g}\|_2 \|\widehat{h}\|_2
    = \|g\|_2 \|h\|_2
  \end{equation*}
  and similarly
  \begin{equation*}
    \||\xi|^m\widehat{u}\|_1
    = \||\xi|^m \widehat{g} \widehat{h}\|_1
    \leq
    \||\xi|^m \widehat{g}\|_2  \|\widehat{h}\|_2
    = \|D^m g\|_2 \|h\|_2.
  \end{equation*}
\end{proof}

\begin{thm}\label{thm:Hm-decay}
  Let the potential $W$ satisfy $\nabla W \in (L^1\cap L^\infty)^N$ and
  let $\rho_0 \in L^1\cap L^{\infty}\cap H^{m}$, $m\geq 1$, with
  $\rho_0$ nonnegative. Consider $\rho$ the solution to \eqref{AE}
  given by Theorems {\rm \ref{global}} and
  {\rm\ref{th:Linftybound}} with data $\rho_0$ . Assume that the solution
  satisfies the $L^2$-decay estimate \eqref{eq:L2decay}. Then there
  exists a constant $C \geq 0$ which depends only on $W$, $\rho_0$,
  $m$ and $N$ such that for all $t\geq 0$
\begin{equation}
    \label{eq:Hm-decay}
    \|D^m\rho_t\|_2 \leq C (t+1)^{-(N/4+m/2)} \, .
\end{equation}
\end{thm}

\begin{proof}
  In this proof we denote by $C$ any nonnegative number that depends
  only on the same quantities as the constant $C$ in the statement. We
  first need to show that for $t$ large enough
  \begin{equation}
    \label{energy-m}
    \frac{d}{dt} \int |D^m\rho|^2 \,dx
    \leq
    -\frac{1}{2} \int | D^{m+1} \rho|^2 \,dx.
  \end{equation}
  For this, apply the operator $D^m$ to the equation \eqref{AE1},
  multiply by $D^m \rho$, and integrate in space. After reordering and
  integration by parts it follows that
  \[
  \frac{d}{dt} \int |D^m \rho|^2 dx = -2\int D^{m} \rho \,D^m\nabla
  \cdot (\rho( \nabla W\ast \rho)) dx -2\int\sum_{j=1}^N |\partial_j
  D^{m} \rho |^2 dx \,.
  \]
  Integrating by parts in the first integral on the right-hand side
  yields
  \[
  \frac{d}{dt} \int |D^m \rho|^2 dx = 2\int \sum_{j=1}^N (\partial_j
  D^{m} \rho) \, D^m (\rho( \partial_j W\ast \rho)) dx -2\int
  \sum_{j=1}^N |\partial_j D^{m} \rho |^2 dx\,.
  \]
  By H\"older and Young's inequalities, we obtain
  \begin{equation*} \label{cosa2} \frac{d}{dt} \int |D^m \rho|^2 dx
    \leq \int | D^m (\rho (W\ast \nabla \rho))|^2 dx -
    \int\sum_{j=1}^n |\partial_j D^{m} \rho |^2 dx\,.
  \end{equation*}
  It follows that
  \begin{equation} \label{cosa3} \frac{d}{dt} \int |D^m \rho|^2 dx
    \leq \int | D^m (\rho (W\ast \nabla \rho))|^2 dx - \int\ | D^{m+1}
    \rho |^2 dx\,,
  \end{equation}
  where we have used that
  \[
  \left(\sum_{j=1}^N |\xi_j|^2\right)^{m+1} | \widehat{u}|^2=
  \left(\sum_{j=1}^N |\xi_j|^2\right)^{m}\sum_{k=1}^N |\xi_k|^2|
  \widehat{u}|^2\,.
  \]
  We need to bound the first term on the right-hand side of the
  inequality (\ref{cosa3}). Using Lemma \ref{lem:Dm-product},
  \begin{equation}
    \label{eq:first-term}
    \int | D^m (\rho (W * \nabla \rho))|^2 \,dx
    \leq
    C \left(
      \| D^m \rho \|_2^2 \|W\|_2^2 \| \nabla \rho\|_2^2
      + \| \rho \|_2^2 \|W\|_2^2 \|D^{m+1} \rho\|_2^2
    \right) =: I + II.
  \end{equation}
  To estimate $I$, we proceed by means of the GNS inequality
  (\ref{eq:GNS-Friedman}) to get
  $$
  \|D^m\rho \|_2^2 \leq
  C\|D^{m+1}\rho\|_2^{\frac{2m}{m+1}}\|\rho\|_2^{\frac{2}{m+1}} \qquad
  \mbox{and}\qquad\|\nabla \rho\|^2_2 \leq C
  \|D^{m+1}\rho\|_2^{\frac{2}{m+1}} \|\rho\|_2^{\frac{2m}{m+1}}
  $$
  which yields
  \begin{equation*}
    I \leq C \|D^{m+1}\rho\|_2^2 \, \|\rho\|_2^2 \, \|W\|_2^2\,.
  \end{equation*}
  Since the term $II$ is already of the same type, from
  \eqref{eq:first-term} we then have
  \begin{equation*}
    \int | D^m (\rho (W * \nabla \rho))|^2 \,dx
    \leq
    C \|D^{m+1}\rho\|_2^2 \, \|\rho\|_2^2 \, \|W\|_2^2
  \end{equation*}
  Since $W\in L^2$ and $\|\rho\|_2 \to 0 $ due to Theorem
  \ref{th:L2decay}, it follows that there exists $T_0$ so that for all
  $t \geq T_0$,
  \begin{equation*}
    \int | D^m (\rho (W * \nabla \rho))|^2 \,dx
    \leq 1/2 \|D^{m+1}\rho\|_2^2\,.
  \end{equation*}
  Using this in \eqref{cosa3}, inequality (\ref{energy-m}) follows for
  $t\geq \Tt = \max{\{T_0,T_1\}}$.

  Now, after taking the Fourier transform and applying Plancherel's
  Theorem it follows from (\ref{energy-m}) that
  \[
  \frac{d}{dt} \int ||\xi|^m\widehat{\rho}|^2 \,d\xi \leq -\frac{1}{2} \int
  ||\xi|^{m+1} \widehat{\rho}|^2 \,d\xi.
  \]
  Now we proceed by Fourier splitting, see \cite{S}. We split the
  frequency domain into $\Ss$ and $\Ss^c$, where
  \[
  \Ss(t) = \{ \xi: |\xi| \leq \Gg(t)\},\;\, \Gg(t) =
  \left(\frac{2k}{t+1}\right)^{1/2},
  \]
  and $k$ is a positive number to be chosen later. By Plancherel's
  identity we obtain
  \begin{equation*}
    \begin{split}
      \frac{d}{dt} \int |\widehat{D^m{\rho}}|^2 \,d\xi
      &\leq
      - \frac12 \int_{\R^N} |\xi|^2|\widehat{D^m{\rho}}|^2 \,d\xi
      \leq
      -\frac{k}{t+1} \int_{\Ss^c(t)}|\widehat{D^m{\rho}}|^2 \,d\xi
      \\
      &\leq
      -\frac{k}{t+1} \int_{\R^N}|\widehat{D^m{\rho}}|^2 \,d\xi +
      \frac{k}{t+1} \int_{\Ss(t)}|\widehat{D^m{\rho}}|^2 \,d\xi.
    \end{split}
  \end{equation*}
  Hence, using Theorem \ref{th:L2decay},
  \begin{align*}
    \frac{d}{dt} \left[ (t+1)^k \|D^m \rho\|^2_2\right]
    &\leq
    k\, (t+1)^{k-1} \int_{\Ss(t)}||\xi|^m \widehat{\rho}|^2 \,d\xi
    \\
    &\leq
    k\, (2k)^m\, (t+1)^{k-1-m} \int_{\Ss(t)} |\widehat{\rho}|^2 \,d\xi
    \leq
    C \|\rho_0\|_{1}^2\,(t+1)^{k-1-m-\frac{N}{2}}.
  \end{align*}
  Choosing now $k > N/2+m$ and integrating on $ [\Tt,t] $ gives the
  desired decay rate \eqref{eq:Hm-decay}. Together with the a
  priori estimate at time $t =\Tt$ obtained in Proposition
  \ref{pr:Lp-bounds}, this concludes the proof.
\end{proof}

\subsection{$L^p$ decay}

We give first the decay for the $L^{\infty}$ norm. The decay of
the solutions in all other $L^p$ spaces follows by interpolation.

\begin{Lemma}
  \label{lem:Lp-decay}
  Let the interaction potential $W$ be such that $\nabla W \in (L^1\cap
  L^\infty)^N$. Let $\rho_0 \in L^1\cap L^{\infty}\cap H^{m+1}$, with
  $m>N/2$, and $\rho_0$ nonnegative. Consider $\rho$ the solution to
  \eqref{AE1}-\eqref{AE2} with data $\rho_0$ given by Theorem {\rm
    \ref{global}} with the properties in Theorems
  {\rm\ref{th:Linftybound}} and {\rm\ref{thm:Hm-decay}} under one of
  the additional smallness assumptions in Theorem
  {\rm\ref{th:L2decay}}. Then there is some constant $C> 0$ such that
  for all $t\geq 0$ and $p\in [2,\infty]$,
\begin{equation}\label{eq:Linfty-decay}
    \|\rho_t\|_p \leq C (t+1)^{-N/2(1-1/p)} \,.
\end{equation}
\end{Lemma}

\begin{proof}
  Let us start by the case $p=\infty$. Using the GNS inequality
  \eqref{eq:GNS-Friedman} with $j=0$, $a=\infty$, $b=2$ and $s=1$, we obtain
  \begin{equation*}
    \|\rho\|_\infty \leq C \| D^m \rho \|_2^{\frac{2N}{2m+N}} \| \rho
    \|_1^{\frac{2m-N}{2m+N}},
  \end{equation*}
  valid for any $m > N/2$. Using the decay, we have for $\| D^m \rho
  \|_2$,
  \begin{equation*}
    \|\rho\|_\infty \leq
    C (1+t)^{-(\frac{m}{2}+\frac{N}{4}) \frac{2N}{2m+N}}
    = C (1+t)^{-\frac{N}{2}}.
  \end{equation*}
  The general case $p \in [2,\infty)$ follows by interpolating $L^p$
  between $L^2$ and $L^{\infty}$.
\end{proof}

\subsection{Decay of $\| x \rho\|_2$}

We need to study the behavior of the norm $\| x \rho \|_2$, as this norm
will appear later in estimates. We begin by studying a moment of
$\rho$:

\begin{Lemma}
  \label{lem:M2-bounded}
  Assume the conditions from Lemma \ref{lem:Lp-decay}. In addition,
  suppose that $|x|^2\rho_0\in L^1$ and that for some $C > 0$,
\begin{equation} \label{eq:x_nabla_W_bounded}
    \big|x\cdot \nabla W (x) \big| \leq C\,,
    \qquad x \in \R^N.
\end{equation}
Then, there is a constant $C = C(\|\rho_0\|_1, N)$ such that
for all $t\geq 0$
\begin{equation}
    \label{eq:M2-bounded}
    \int |x|^2 \rho(t,x) \,dx
    \leq
    \int |x|^2 \rho_0(x)\,dx + C t
\qquad \mbox{and} \qquad
    \| x \rho_t \|_2^2 \leq C (1+t)^{1 -\frac{N}{2}}\,.
\end{equation}
\end{Lemma}

\begin{proof}
  Multiplying \eqref{AE1} by $|x|^2$ and integrating, we get
  \begin{align*}
    \frac{d}{dt} \int |x|^2 \rho(x) \,dx &= 2N \|\rho_0\|_1 - 2
    \int \rho (x \cdot (\nabla W * \rho))\,dx\\
    &= 2N \|\rho_0\|_1 - \int \!\!\!\int (x-y) \rho(x)\rho(y) \nabla
    W(x-y)\,dx\,dy \leq 2N \|\rho\|_1 + C\|\rho_0\|_1^2,
  \end{align*}
  which establishes \eqref{eq:M2-bounded}. Now, using
  \eqref{eq:Linfty-decay} and \eqref{eq:M2-bounded}, we get
  \begin{equation*}
    \| x \rho \|_2^2
   =
    \int_{\R^N} |x|^2 \rho(x)^2 \,dx
    \leq
    \|\rho\|_\infty \int_{\R^N} |x|^2 \rho(x) \,dx
    \leq
    C (1+t)^{1 -\frac{N}{2}}\, .
  \end{equation*}
\end{proof}


\section{Asymptotic simplification towards the Heat Equation}
\label{sec:self-similarity}

In this section we will show that the flow of the
aggregation-diffusion equation \eqref{AE} behaves like
the heat equation for large times provided we are under conditions for
which the $L^2$ decay of Theorem \ref{th:L2decay} holds. We  prove
it by two different arguments; though the results obtained by
following these two strategies are similar, we have kept both of them
because they give a better understanding of the behavior of the
equation. The proof in section \ref{sec:direct} is based on direct
estimates of the bilinear form \eqref{bilinear}. It is quite
straightforward and gives a better result in the one-dimensional
($N=1$) case. On the other hand, the argument given in section
\ref{sec:rescaling} is based on the self-similar change of variables
and entropy arguments as in \cite{CT,CF}, which we briefly describe
now.

We consider the standard self-similar scaling \cite{CT0,CT} as is
usually done for the heat equation to pass to the corresponding
Fokker-Planck equation:
\begin{equation}
  \label{eq:scaling}
  f(s,y) = e^{Ns}\rho\left(\frac12 (e^{2s}-1), e^{s} y\right),
\end{equation}
or, equivalently,
\begin{equation}
  \label{eq:scaling-inverse}
  \rho(t,x) = (1+2t)^{-N/2} f\left( \frac12 \log (1+2t), (1+2t)^{-1/2} x\right).
\end{equation}
Then $f$ satisfies the equivalent rescaled equation
\begin{equation}
  \label{eq:rescaled-PDE}
  \partial_s f
  =
  \nabla_y \cdot (yf) + \Delta_y f
  + \nabla_y \cdot (f (\nabla_y \widetilde{W} * f)),
\end{equation}
with initial data $f(0,y)=\rho_0(y)$ if and only if $\rho$ is a
solution to \eqref{AE}. Here, we write
\begin{equation}
  \label{eq:W_tilde}
  \widetilde{W}(s,y) := W(y e^{s}),
  \quad \text{ so } \quad
  \nabla_y \widetilde{W}(s,y) := e^{s} \nabla W(y e^{s}).
\end{equation}
We consider the entropy functional
\begin{equation}
  \label{eq:H}
  H[f] := \int \left( f \log f + \frac{|y|^2}{2} f\right)\,dy.
\end{equation}
We show below that $H[f]$ converges exponentially to $0$. The
convergence  for $\rho$ obtained in this way is stronger, and requires stronger
conditions on the initial data. Also, observe that in self-similar
variables the nonlinear term transforms to the time-dependent term
involving $\widetilde{W}$, where the asymptotic simplification becomes
apparent.

\medskip The results in this section are gathered in the following
theorem. We recall that $G$ is the fundamental solution of the heat
equation, defined in \eqref{eq:heat-kernel}.

\begin{thm}
  \label{th:asymptotic-main}
  Let the interaction potential $W\in L^1 \cap L^2$ such that $\nabla
  W \in (L^1\cap L^\infty)^N$. Consider $\rho$ the solution to
  \eqref{AE1}-\eqref{AE2} with data $\rho_0 \in L^1\cap L^{\infty}\cap
  H^2$, $\rho_0 \geq 0$, constructed in Theorem {\rm \ref{global}}. We
  denote $M := \|\rho_0\|$. Assume also one of the smallness
  conditions in Theorem {\rm\ref{th:L2decay}}.
  \begin{enumerate}
  \item[i)]
    Then there exists $C>0$ such that:
    \begin{enumerate}
    \item For $N = 1$,
      \begin{equation}
        \label{eq:N1-L1-heat-convergence}
        \big\|\rho_t- M G(t) \big\|_1
        \leq \frac{C \log t}{\sqrt{t}}
        \quad \text{ for all } t \geq 1.
      \end{equation}
    \item For $N \geq 2$,
      \begin{equation}
        \label{eq:L1-heat-convergence}
        \big\| \rho_t- M G(t) \big\|_1
        \leq \frac{C}{\sqrt{t}}
        \quad \text{ for all } t \geq 1.
      \end{equation}
    \end{enumerate}
  \item[ii)]
    Assume further that $|x|^2\rho_0\in L^1(\R^N)$. Then there exists
    $C>0$ such that:
    \begin{enumerate}
    \item For $N = 1$,
      \begin{equation}
        \label{eq:H-decay-N1}
        H[f] \leq C \, e^{-s}
        \quad \text{ for all } s > 0.
      \end{equation}
    \item For $N =2$,
      \begin{equation}
        \label{eq:H-decay-N2}
        H[f] \leq C \, (1+s)\, e^{-2s}
        \quad \text{ for all } s > 0.
      \end{equation}
    \item For $N \geq 3$,
      \begin{equation}
        \label{eq:H-decay}
        H[f] \leq C \, e^{-2s}
        \quad \text{ for all } s > 0.
      \end{equation}
    \end{enumerate}
  \end{enumerate}
\end{thm}

\begin{Remark}
  We notice that under the above hypotheses Theorems
  {\rm\ref{th:Linftybound}} and {\rm\ref{th:L2decay}} apply to the
  solution $\rho$ in the statement.
\end{Remark}

\begin{Remark}
It is well-known that \eqref{eq:H-decay} directly implies, through
the Csisz\'{a}r-Kullback inequality {\rm\cite{To99,2AMTU}} and the
change of variables \eqref{eq:scaling}, that
\begin{equation*}
    \big\| \rho_t- M G(t) \big\|_1
    \leq C \, t^{-1/2} \quad \text{ for all } t\geq 0 \,.
\end{equation*}
This gives the same order of convergence as
\eqref{eq:L1-heat-convergence} for any $N \geq 3$. However, the
corresponding results for $N = 1$ or $2$ are weaker than
\eqref{eq:N1-L1-heat-convergence}. In summary, point ii) above,
which uses entropy/entropy dissipation tools, gives a worse rate
of convergence than point i), whose proof uses a direct argument.
However, one has to take into account that point ii) bounds a
nonlinear quantity, the logarithmic entropy, essentially a
quadratic functional at first order in expansion around the
Maxwellian. Therefore, it is not surprising that we get a weaker
result by estimating a quadratic functional. In fact, a direct
argument in the spirit of point i) estimating the difference in
$L^2$ of the solution with respect to the fundamental solution of
the heat equation leads to similar rates of convergence as the
entropy argument.
\end{Remark}

\subsection{Direct argument}
\label{sec:direct}

Let us prove point \emph{i)} of Theorem
\ref{th:asymptotic-main}. We first give a technical lemma that will be
used below:

\begin{Lemma}
  \label{lem:truncated-beta}
  For $\alpha \geq 1$, we have the following bound

 \[   \int_0^t (t-s)^{-1/2}
    (1+s)^{-\alpha}
    \,ds    \leq
    \left\{
  \begin{array}{l l}
     \frac{C}{t^{1/2}} & \quad \text{if $\alpha >1$ }\\
    \frac{C}{t^{1/2}} \log(t)& \quad \text{if $\alpha=1$ }\\
  \end{array} \right.
 \quad \text{ for all } t \geq 1,
 \]
  where $C > 0$ is some number which depends only on $\alpha$.
\end{Lemma}
\begin{proof}

\[  \int_0^t (t-s)^{-1/2}  (1+s)^{-\alpha}
    \,ds =  \int_0^{t/2}(1+s)^{-\alpha}  \,ds + \int_{t/2}^t (t-s)^{-1/2}  (1+s)^{-\alpha} \,ds =I +II\]
To estimate I we proceed as follows, bounding $(t-s)^{-1/2}$ by $(t/2)^{-1/2}$:
 \[ I=  \int_0^{t/2} (t-s)^{-1/2}
    (1+s)^{-\alpha}
    \,ds    \leq
    \left\{
  \begin{array}{l l}
     \frac{C}{t^{1/2}} & \quad \text{if $\alpha >1$ }\\
    \frac{C}{t^{1/2}} \log(t+1)& \quad \text{if $\alpha=1$ }\\
  \end{array} \right.
 \quad \text{ for all } t \geq 1,
 \]
The estimate of II follows by
 \[ II=  \int_{t/2}^t (t-s)^{-1/2}
    (1+s)^{-\alpha}
    \,ds    \leq   (1+t/2)^{-\alpha}\int (t-s)^{-1/2} \,ds \leq  \frac{C}{t^{1/2}}\]

Adding the estimates for I and II  gives the conclusion of the lemma.
\end{proof}

\begin{proof}[Proof of point \emph{i)} of Theorem
  \ref{th:asymptotic-main}]
  Taking into account \eqref{eq:heat-kernel-dx-bound}, we get
  $\|\nabla G(t)\|_1 \leq C_N t^{-1/2}$, for some constant $C_N > 0$
  depending only on the dimension. Using this estimate on the second
  term in the right-hand side of the Duhamel formula
  \eqref{eq:Duhamel}, we get
  \begin{align*}
    h(t) &:=
    \int \int_0^t \left|\Big( \nabla G(t-s,\cdot)
    \right.*\left. (\rho(s,\cdot)(\nabla W \ast\rho)(s,\cdot))
      \Big)(x)\right| \,ds \,dx
    \\
    &\leq \int_0^t \| \nabla G(t-s,\cdot) \|_1 \ \| \rho(s,\cdot)
    (\nabla W \ast\rho)(s,\cdot) \|_1 \,ds
    \\
    &\leq C_N \int_0^t (t-s)^{-1/2} \ \| \rho(s,\cdot)\|_2 \| (\nabla
    W * \rho)(s,\cdot) \|_2 \,ds
    \\
    &\leq C_N \int_0^t (t-s)^{-1/2} \ \| \rho(s,\cdot)\|_2 \| W \|_1
    \| \nabla \rho (s,\cdot) \|_2 \,ds
    \\
    &\leq C \| W \|_1 \int_0^t (t-s)^{-1/2} \ (1+s)^{-(N/2 + 1/2) }
    \,ds,
  \end{align*}
  where we have used Theorems \ref{th:L2decay} and \ref{thm:Hm-decay}
  for $m=1$. Now, using Lemma \ref{lem:truncated-beta} for $N\geq 2$,
  we obtain
  \begin{equation*}
    h(t) \leq
    C \| W \|_1 t^{-1/2}
    \qquad (t \geq 1),
  \end{equation*}
  for some constant $C > 0$ which depends on the dimension $N$. From
  \eqref{eq:Duhamel}, this gives
  \begin{equation*}
    \| \rho_t - e^{t\Delta} \rho_0\|_1
    \leq
    C \| W \|_1 t^{-1/2}
    \quad
    \text{ for all } t \geq 1.
  \end{equation*}
  Using the known asymptotic behavior of the heat equation, for
  instance \cite{DZ}, the claim of the Theorem follows. In the case $N
  = 1$, the same reasoning using the second bound in Lemma
  \ref{lem:truncated-beta} gives
  \begin{equation*}
    \| \rho_t - e^{t\Delta} \rho_0\|_1
    \leq
    C \| W \|_1 t^{-1/2} \log t
    \quad
    \text{ for all } t \geq 1,
  \end{equation*}
  which proves the result.
\end{proof}

\subsection{Entropy argument} \label{sec:rescaling}

We prove point \emph{ii)} of Theorem \ref{th:asymptotic-main}:

\begin{proof}
  In addition to the entropy \eqref{eq:H}, we define as usual the
  entropy dissipation of the linear Fokker-Planck equation as
  \begin{equation*}
    D[f] := \int f \left| \nabla \left(\frac{|y|^2}{2} + \log f
      \right) \right|^2 \,dy\,.
  \end{equation*}

  As it is classically known \cite{CMV}, the evolution of the free
  energy for the equation \eqref{eq:rescaled-PDE} can be obtained as
  \begin{multline*}
    \frac{d}{ds} \left[ H[f] + \frac{1}{2} \int \!\!\! \int
      \widetilde{W}(s,x-y) f(x) f(y) \,dx\,dy \right]
    \\
    = - \int f \left| \nabla \left( \frac{|y|^2}{2} + \log f +
        \widetilde{W}*f \right) \right|^2 \,dy+ \frac{1}{2} \int
    \int \partial_s \widetilde{W}(s,x-y) f(x) f(y) \,dx\,dy,
  \end{multline*}
  and expanding part of the square in the first part yields the term
  $D[f]$ on the right-hand side of the equation
  \begin{multline}
    \label{eq:dtH2}
    \frac{d}{ds} \left[ H[f] + \frac{1}{2} \int \!\!\! \int
      \widetilde{W}(s,x-y) f(x) f(y) \,dx\,dy \right] = - D[f] - \int
    f \left| \nabla \widetilde{W}*f \right|^2\,dy
    \\
    + 2 \int f\, \nabla \left( \frac{|y|^2}{2} + \log f \right) \cdot
    \left( \nabla \widetilde{W}*f \right)\,dy + \frac{1}{2} \int
    \!\!\! \int \partial_s \widetilde{W}(s,x-y) f(x) f(y) \,dx\,dy
    \\
    =: - D[f] + T_2 + T_3 + T_4.
  \end{multline}
  Recall  that due to the classical Logarithmic Sobolev inequality
  \cite{Tos,To}, we have
  \begin{equation}
    \label{eq:log-sobolev}
    2 H[f] \leq D[f].
  \end{equation}
  We show that terms other than $-D[f]$ decay in time at least like
  $e^{-s}$. For the term $T_3$ we have
  \begin{equation}
    \label{eq:T3}
    T_3
    =
    2 \int f\,y\cdot(\nabla \widetilde{W}*f)\,dy
    + 2 \int \nabla f \cdot (\nabla \widetilde{W}*f)\,dy
    =: T_{31} + T_{32},
  \end{equation}
  We prove that both $T_{31}$ and $T_{32}$ decay exponentially: for
  $T_{31}$, using the Cauchy-Schwarz and the Young inequalities,
  \begin{equation}
    \label{eq:T31}
    T_{31} \leq
    2 \norm{y f}_2 \|\widetilde{W} * \nabla f\|_2
    \leq
    2 \norm{y f}_2 \norm{\nabla f}_2 \|\widetilde{W}\|_1
    =
    2 C e^{-Ns} \norm{y f}_2 \norm{\nabla f}_2
    \leq
    C e^{-Ns}\,.
  \end{equation}
  Here, we used that $\| \W \|_1 = e^{-Ns} \norm{W}_1$ due to
  \eqref{eq:W_tilde}. The last step follows since both $\norm{y f}_2$
  and $\norm{\nabla f}_2$ are uniformly bounded for all times by
  Theorem \ref{thm:Hm-decay}, \eqref{eq:M2-bounded} and the change of
  variables \eqref{eq:scaling}-\eqref{eq:scaling-inverse}. By a
  similar argument for $T_{32}$, we have
  \begin{equation}
    \label{eq:T32}
    T_{32}
    =
    2 \int \nabla f \cdot(\widetilde{W} * \nabla f)
    \leq
    2 \norm{\nabla f}_2^2 \| \widetilde{W} \|_1
    = 2 C e^{-Ns} \norm{\nabla f}_2^2
    \leq C e^{-Ns}.
  \end{equation}
  Using that $\norm{\nabla f}_2$ is uniformly bounded in time
  due to Theorem \ref{thm:Hm-decay} through the change of variables
  \eqref{eq:scaling}-\eqref{eq:scaling-inverse}. Note that, $T_2 \leq 0$.
  For $T_4$ we  use that $\partial_s \W (s,y) = y \cdot
  \nabla_y \W(s,y)$, the Cauchy-Schwarz and Young inequalities as
  above, to get
  \begin{equation}
    \label{eq:T4}
    T_4
    =
    \frac{1}{2} \int f (\partial_s \W*f)\,dy
    =
    \frac{1}{2} \int f \,y\cdot (\nabla \W * f)\,dy
    = \frac14 T_{31} \leq
    C e^{-Ns} \quad \text{ for } s \geq 0.
  \end{equation}
  On the other hand, the integral on the left-hand side of
  \eqref{eq:dtH2} can be bounded as follows
  \begin{equation}
    \label{eq:Tdt}
    \left|\frac{1}{2} \int\! \! \! \int \widetilde{W}(s,x-y) f(x) f(y)
      \,dx\,dy\right|
    \leq \frac{1}{2} \int f (|\widetilde{W}|*f)\,dy
    \leq
    \frac{1}{2} \|f\|_2^2 \|\widetilde{W}\|_1
    \leq C e^{-Ns}.
  \end{equation}
  Hence from \eqref{eq:dtH2}, using \eqref{eq:log-sobolev},
  \eqref{eq:T3}, \eqref{eq:T31}, \eqref{eq:T32}, \eqref{eq:T4}, and
  \eqref{eq:Tdt}, we obtain for all $s \geq 0$ that
  \begin{align*}
    \frac{d}{ds} \left( H[f] + \frac{1}{2} \int f
      (\widetilde{W}*f)\,dy \right) &\leq -2 \left( H[f] + \frac{1}{2}
      \int f (\widetilde{W}*f)\,dy \right) + C e^{-Ns} + \int f
    (\widetilde{W}*f)\,dy
    \\
    &\leq -2 \left( H[f] + \frac{1}{2} \int f (\widetilde{W}*f)\,dy
    \right) + C e^{-Ns}\,.
  \end{align*}
  We remark that we have kept the term $\frac{1}{2} \int \!\!\! \int
  \widetilde{W}(s,x-y) f(x) f(y) \,dx\,dy$ inside the time derivative
  to avoid the appearance of the time derivative of $f$. Hence, from this
  differential inequality we deduce that there is some constant $C$
  which depends only on the initial condition $f_0$ such that
  \begin{equation*}
    H[f] + \frac{1}{2} \int f (\widetilde{W}*f)
    \leq
    C e^{-\min\{N,2\}s}
    \quad \text{ for $N \neq 2$},
  \end{equation*}
  with $C (1+s) e^{-2s}$ on the right-hand side for $N=2$. This
  implies, using again \eqref{eq:Tdt},
  \begin{equation*}
    H[f] \leq C e^{-\min\{N,2\}s}
    \quad \text{ for $N \neq 2$},
  \end{equation*}
  for some other number $C$ depending only on $f_0$. Again, the right
  hand side should be substituted by $C (1+s) e^{-2s}$ for $N=2$. The
  final steps are just classical implications of this entropy
  control. Denoting the Maxwellian by
  $$
  \mathcal{M}(y) = (2\pi)^{-N/2} \exp\left(-\frac{|y|^2}2 \right).
  $$
  By means of the Csisz\'{a}r-Kullback inequality, see \cite{To,To99,2AMTU}, we get
  $$
  \|f(s)- M \mathcal{M}(s)\|_1^2
  \leq
  C (H(f(s))-H(\mathcal{M}(s))) \leq C e^{-\min\{N,2\}s},
  $$
  from which the announced result follows through the change of
  variables \eqref{eq:scaling}-\eqref{eq:scaling-inverse}.
\end{proof}


\subsection*{Acknowledgments}
J.~A.~Ca\~nizo and J.~A.~Carrillo are partially supported by the
projects MTM2011-27739-C04 from DGI-MICINN (Spain) and
2009-SGR-345 from AGAUR-Generalitat de Catalunya. M.~Schonbek is
partially supported by the NSF Grant DMS-0900909. M. Schonbek was
also supported by the sabbatical program of the MEC-Spain Grant
SAB2009-0024.


\bibliographystyle{plain}

\end{document}